\documentclass[12pt,reqno]{amsart}

\usepackage{amssymb}
\usepackage{verbatim }\usepackage[all,cmtip]{xy}

\newtheorem{theorem}{Theorem}[section]
\newtheorem{proposition}[theorem]{Proposition}
\newtheorem{lemma}[theorem]{Lemma}

\numberwithin{equation}{section}

\newcommand{\Coh}{\mathrm{H}}
\newcommand{\TX}{\mathrm{T}X}
\newcommand{\KX}{\mathrm{K}_X}
\newcommand{\cTX}{\mathrm{T}\mathcal{X}}
\newcommand{\TlogD}{\mathrm{T}X(-\log D)}
\newcommand{\TXD}{\mathrm{T}X(-{D})}
\newcommand{\Atlog}{\mathrm{At}_{D}(E)}
\newcommand{\AtEL}{\mathrm{At}_{D}(E,L)}
\newcommand{\cTlogD}{\mathrm{T}\mathcal{X}(-\log {\mathcal{D}})}
\newcommand{\cAtlog}{\mathcal{A}t_{\mathcal{D}}(\mathcal{E})}
\newcommand{\Tan}{\mathrm{T}}
\newcommand{\Par}{\mathcal{T}}
\newcommand{\para}{t}
\newcommand{\degree}{\mathrm{deg}}

\begin{document}

\baselineskip=15pt

\title[Isomonodromic deformations and very stable bundles]{Isomonodromic deformations and very 
stable vector bundles of rank two}

\author[I. Biswas]{Indranil Biswas}

\address{School of Mathematics, Tata Institute of Fundamental
Research, Homi Bhabha Road, Mumbai 400005, India}

\email{indranil@math.tifr.res.in}

\author[V. Heu]{Viktoria Heu}

\address{Institut de Recherche Math\'ematique Avanc\'ee, Univesit\'e de Strasbourg, 7 rue 
Ren\'e-Descartes, 67084 Strasbourg Cedex, France}

\email{heu@math.unistra.fr}

\author[J. Hurtubise]{Jacques Hurtubise}

\address{Department of Mathematics, McGill University, Burnside
Hall, 805 Sherbrooke St. W., Montreal, Que. H3A 0B9, Canada}

\email{jacques.hurtubise@mcgill.ca}



\thanks{The first author is supported by a J. C. Bose Fellowship. The second author is supported 
by ANR-13-BS01-0001-01 and ANR-13-JS01-0002-01.}

\date{}

\begin{abstract}
For the universal isomonodromic deformation of an irreducible logarithmic rank two connection 
over a smooth complex projective curve of genus at least two, consider the family of 
holomorphic vector bundles over curves underlying this universal deformation. In a previous 
work we proved that the vector bundle corresponding to a general parameter of this family is 
stable. Here we prove that the vector bundle corresponding to a general parameter is in fact 
very stable (it does not admit any nonzero nilpotent Higgs field).
\end{abstract}

\maketitle

\tableofcontents

\section{Introduction}

Let $E\longrightarrow X$ be a rank two holomorphic vector bundle over a compact Riemann surface
$X$ of genus $g$. The \emph{Segre-invariant} of $E$ is defined as
$$
\kappa (E)\,:=\,\min_L (\degree(E)-2\degree(L))\, ,
$$
where the minimum is taken over all holomorphic line subbundles $L$ of $E$. This $\kappa (E)$ is
bounded above by $g$ \cite{Na}. If
$\kappa(E)\,>\,0$ then $E$ is called \textit{stable}, and it is called
\emph{maximally stable} if $\kappa(E)\,>\,g-2$ with $g\, \geq\, 2$. Now let $\pi \,:\, 
\mathcal{X}\,\longrightarrow\, \Par$ be a holomorphic family of compact Riemann 
surfaces of genus $g$, and let $\mathcal{E}$ be a holomorphic vector bundle over
$\mathcal{X}$. For every $\para \in \Par$, set $X_\para\,:=\,\pi^{-1}(\para)$ and 
$E_\para\,:=\,\mathcal{E}\vert_{X_\para}$. We shall further denote
\begin{equation}\label{filtBk}
\Par_k\,:=\,\{\para \,\in\, \Par~\mid ~\kappa(E_\para)\,\leq \,k\}
\end{equation}
for every $k\,\in\, \mathbb{Z}$. From a theorem of Maruyama and Shatz, \cite{Maru}, \cite{Sh}, 
it follows that $\Par_k$ is a closed analytic subset of $\Par$ for each $k$. The codimension of 
$\Par_k$ in $\Par$ will be denoted by $\mathrm{codim}( \Par_{k} , \,\Par)$. By convention,
the empty set has infinite codimension. The 
Riemann--Hilbert type question addressed in \cite{He}, \cite{BHH}, rephrased for 
the particular case of rank $2$ vector bundles, would be the following:

\emph{What is the maximal Segre-invariant for a family of
holomorphic vector bundles that can be endowed with 
a flat logarithmic connection, inducing a prescribed logarithmic connection on a central 
parameter $\para_0\,\in \, \Par$}?

Indeed, the general theme of \cite{BHH}, continued to be pursued here, is that the isomonodromic 
deformations provide a good ``transverse'' family of deformations of a vector bundle, in which 
exceptional behaviors, such as instability, seem to occur with an appropriate codimension. Here 
we consider the stronger condition of very stability. The very stable bundles were introduced by 
Laumon \cite{La} and play an important role in the study of the Hitchin systems. They are 
generic but turn out to be extraordinarily hard to produce explicitly. In fact, the existence of 
very stable bundles is a landmark result \cite{La}; the dominance of the Hitchin map was 
obtained as a consequence of the existence of very stable bundles.

Let $D\,=\,\{x_1,\, \cdots,\, x_n\}\, \subset\, X$ be a finite subset of $n\,\geq\, 0$
distinct points. The corresponding reduced divisor $\sum_{i=1}^n x_i$ on $X$ will also be denoted by
$D$. Let $\delta$ be a logarithmic connection, singular over $D$, on a holomorphic vector bundle
$E\,\longrightarrow\, X$ (see Section \ref{Sec:Atiyah} for its definition). An
isomonodromic deformation of the quadruple $(X,\, D,\, E,\,\delta)$ is given by the following data:
\begin{itemize}
\item a holomorphic family of vector bundles $\mathcal{E}\,\longrightarrow\,
\mathcal{X}\,\longrightarrow\,\Par$ as above,

\item a divisor $\mathcal{D}$ on $\mathcal{X}$ given by the sum of the images of 
$n$ disjoint holomorphic sections $\Par\,\longrightarrow\, \mathcal{X}$,

\item a flat logarithmic connection $\nabla$ on $\mathcal{E}$ singular over 
$\mathcal{D}$, and

\item an isomorphism of pointed curves $f \,:\,(X,\,D)\,\longrightarrow\,(X_{\para_0}
,\, D_{\para_0})$, where $\para_0\,\in\, \Par$ is a base point, together with a holomorphic
isomorphism $\psi\,:\, E\,\longrightarrow
\, f^*\mathcal{E}$ of vector bundles such that $f^*\nabla \,=\,\psi_*(\delta)$.
\end{itemize}

Any quadruple $(X,\, D,\, E,\,\delta)$ as above with $3g-3+n\,>\,0$ admits a
universal isomonodromic deformation, satisfying a universal property for
the isomonodromic deformations as above, with respect to a germification 
$(\Par,\,\para_0)$ of the parameter space \cite{He1}. This is a consequence of the 
logarithmic Riemann--Hilbert correspondence \cite{De} combined with Malgrange's 
lemma \cite{Mal}. For this universal isomonodromic deformation, the parameter space 
is the Teichm\"uller space $\mathrm{Teich}_{g,n}$, with a central parameter 
corresponding to $(X,\,D)$, and the family of pointed curves $(\mathcal{X},\,\mathcal{D})$
over it being the universal Teichm\"uller curve. We recall the main theorem of \cite{He}.

\begin{theorem}[\cite{He}]\label{BHHrank2}
Let $(X,\, D,\, E,\,\delta)$ be as above with $3g-3+n\,>\,0$ and $\delta$ irreducible. Let
$\mathcal{E}\,\longrightarrow\, (\mathcal{X},\,\mathcal{D})\,\longrightarrow\,\Par$ be
the holomorphic family of
vector bundles over pointed curves underlying the universal isomonodromic
deformation of $(X,\, D,\, E,\,\delta)$. Then for the filtration
$$\cdots \,\subset\, \Par_{k}\,\subset\, \cdots \,\subset\, \Par_{g-1}
\,\subset\, \Par_{g}\,=\,\Par$$
by closed analytic subsets defined in \eqref{filtBk}, the inequality
$$\mathrm{codim}( \Par_{k} , \,\Par) \,\geq \,g-1-k$$
holds for all $k$, in particular $\mathrm{codim}( \Par_{g-2} , \,\Par) \,> \,0$,
implying that the vector bundle $E_\para$ is maximally stable for general $\para\,\in\, \Par$. 
\end{theorem}

Let again $E$ be a holomorphic vector bundle on $X$ of rank two. A {\it Higgs field} on $E$ is a 
holomorphic section of $\text{End}(E)\otimes K_X\, =\, E\otimes E^*\otimes K_X$, where $K_X$ 
denotes the holomorphic cotangent bundle of $X$. Given a Higgs field $\theta\ \in\, \Coh^0(X,\, 
\text{End}(E)\otimes K_X)$, we have
$$
\text{trace}(\theta^i)\, \in\, \Coh^0(X,\, K^{\otimes i}_X)\, , \ \ i\,=\, 1,\, 2\, .
$$
A Higgs field $\theta$ on $E$ is called \textit{nilpotent} if
$\text{trace}(\theta)\,=\,0\,=\, \text{trace}(\theta^2)$, or equivalently, if $\theta^2
\,=\, 0$; the nilpotent cone plays a very important role in Geometric Langlands program
(see \cite{Fr}, \cite{KW}, \cite{GW}, \cite{DP} and references therein).

A holomorphic vector bundle $E$ of rank two is called \textit{very stable} if it does not admit 
any nonzero nilpotent Higgs field. For any $g\,\geq\,1$, there are maximally stable bundles 
which are not very stable. But a very stable vector bundle $E$ is always maximally stable. 
Indeed, if $L\, \subset\, E$ is a holomorphic line subbundle with $\degree(E)-2\cdot 
\degree(L)\,\leq\, g-2$, then
$$
\chi(\text{Hom}(E/L,\, L\otimes K_X)) \,=\, h^0(\text{Hom}(E/L,\, L\otimes K_X))-
h^1(\text{Hom}(E/L,\, L\otimes K_X))
$$
$$
\,=\, \text{degree}(\text{Hom}(E/L,\, L\otimes K_X))\,=\,
-(\degree(E)-2\cdot \degree(L))+g-1 \,\geq\, 1\, ,
$$
so $\Coh^0(X,\, \text{Hom}(E/L,\, L\otimes K_X))\,\not=\, 0$.
Now, for any nonzero element $$\theta'\, \in\,
\Coh^0(X,\, \text{Hom}(E/L,\, L\otimes K_X))\, ,$$ the Higgs field $\theta$ on $E$ defined by the
composition
$$
E\, \longrightarrow\, E/L \, \stackrel{\theta'}{\longrightarrow}\, L\otimes K_X
\, \hookrightarrow\, E\otimes K_X
$$
has the property that $\theta^2\,=\, 0$. By a similar argument
one can show that $E$ is very stable if and only if 
$\Coh^0(X,\, \text{Hom}(E/L,\, L\otimes K_X))\,=\,0$ for every
holomorphic line subbundle $L$ of $E$.

{}From now on, we are going to assume that the genus $g$ of $X$ is at least two. Note
that then the maximally stable bundles are stable. Let $\mathcal{E}\, \longrightarrow\,
\mathcal{X}\, \longrightarrow\,\Par$ be a holomorphic family of rank two vector bundles over curves of genus
$g$. The subset
\begin{equation}\label{Parnil}
\Par^{nil}\,:=\,\{\para\,\in\, \Par~\mid ~E_\para\longrightarrow X_\para 
\textrm{ is not very stable}\}
\end{equation}
is a closed analytic subset of $\Par$ \cite{La}. The following is a natural question to
ask:

\emph{Is the subset $\Par^{nil}$ of the parameter space of a given universal
isomonodromic deformation proper?}

Our aim here is to prove the following stronger version of Theorem \ref{BHHrank2}:

\begin{theorem}\label{thm1}
Let $(X,\, D,\, E,\,\delta)$ be as above with $g\, \geq\, 2$. Assume that $\delta$ is irreducible. Let
$\mathcal{E}\, \longrightarrow\,(\mathcal{X},\,\mathcal{D}) \, \longrightarrow\,\Par$
be the holomorphic family
of vector bundles over pointed curves underlying the universal isomonodromic deformation of $(X,\, D,\,
E,\,\delta)$. Then the closed analytic subset $ {\Par^{nil}}$, defined in \eqref{Parnil}, is a
proper closed subset of $\Par$. In particular, the vector bundle $E_\para$ is very stable for general
$\para\,\in\, \Par$. 
\end{theorem}

In view of Theorem \ref{BHHrank2}, it is enough to consider the case where the 
holomorphic vector bundle $E$ corresponding to the central parameter is stable but not very 
stable. Following Donagi and Pantev \cite{DP}, stable vector bundles that are not 
very stable will be called \emph{wobbly} vector bundles.

Let $\mathcal{E}\, \longrightarrow\, \mathcal{X}\, \longrightarrow\, \Par$ be a 
family of wobbly rank two vector bundles over compact Riemann surfaces of genus 
$g$. Then there is a polydisc $\mathcal{U}$ in $\Par$ and a section $\Theta$ of 
$(\mathrm{End}(\mathcal{E})\otimes\Omega^1_{\mathcal{X}} )\vert_{\mathcal{U}}$ 
such that for each $\para \,\in\, \mathcal{U}$, the restriction 
$\Theta|_{X_\para}$ defines a nonzero nilpotent Higgs field on $E_\para$. Indeed, 
since one has a family of stable vector bundles on $X$, the family of possible 
Higgs fields for $\mathcal E$ is a holomorphic vector bundle $V$ over $\Par$. The constraint 
$\theta^2\,=\, 0$ of nilpotence then gives us a cone $\mathcal N$ in $V$, which 
maps surjectively onto $\Par$ since the family is wobbly. One can then choose a 
smooth point of $V$ for which the tangent plane to $\mathcal N $ surjects to that 
of $\Par$, and get our section $\Theta$ from there.
 
For that reason, in order to prove Theorem \ref{thm1}, we may use deformation 
theory of stable nilpotent Higgs bundles over curves. A {\it Higgs bundle} on $X$ 
is a pair of the form $(E,\, \theta)$, where $E$ is a holomorphic vector bundle on 
$X$ of rank two and $\theta$ is a Higgs field on $E$. Such a Higgs bundle $(E,\, 
\theta)$ is called \emph{nilpotent} if the Higgs field $\theta$ is nilpotent; if 
$2\cdot \degree (L)\, <\, \degree(E)$ for every holomorphic line subbundle $L\, 
\subset\, E$ with $\theta(L) \,\subset\, L\otimes \KX$, then $(E,\, \theta)$ is 
called \textit{stable}, so $(E,\, \theta)$ is stable if $E$ is stable. More 
specifically, we are going to elaborate, in Section \ref{Sec:DefNilHiggs}, the deformation
theory of stable nilpotent Higgs bundles over pointed curves.

\section{Isomonodromic deformations}

We begin by recalling a few results from Sections 2 and 4 in \cite{BHH} that would be 
needed later; note that in \cite{BHH} the Riemann surface is denoted $X_0$ instead of $X$.

\subsection{Logarithmic connections}\label{Sec:Atiyah}

Let $(E,\,X,\,D)$ be as in the introduction. Let 
$\text{Diff}^i(E,\,E)$ denote the holomorphic vector bundle on $X$ defined by the 
sheaf of differential operators, of order at most $i$, from the sheaf of 
holomorphic sections of $E$ to itself. We have a short exact sequence
\begin{equation}\label{e6}
0\, \longrightarrow\, \text{Diff}^0(E,\,E)\,=\, \text{End}(E)\, \longrightarrow\,
\text{Diff}^1(E,\,E)\, \stackrel{\sigma_1}{\longrightarrow}\, \TX\otimes \text{End}(E)
\, \longrightarrow\, 0\, ,
\end{equation}
where $\sigma_1$ is the symbol homomorphism. Define
$$
\Atlog\, :=\, \sigma^{-1}_1(\TX(-\log D) \otimes
\text{Id}_E)\, \subset\, \text{Diff}^1(E,\,E)\, .
$$
Since $X$ is one--dimensional, we have $\TX(-\log D)\,=\, \TXD\,:=\,\TX\otimes {\mathcal O}_X(
-D)$. From \eqref{e6} it follows that $\Atlog$ fits in an exact sequence of
holomorphic vector bundles on $X$
\begin{equation}\label{e7}
0\, \longrightarrow\, \text{End}(E)\, \longrightarrow\,
 \Atlog\, \stackrel{\sigma}{\longrightarrow}\, \TX(-\log D) \, \longrightarrow\, 0\, ,
\end{equation}
where $\sigma$ is the restriction of the surjection $\sigma_1$ in \eqref{e6}.

A \textit{logarithmic connection} on $E$ singular over $D$ is a holomorphic homomorphism
$$
\delta\, :\, \TX(-\log D) \, \longrightarrow\, \Atlog
$$
such that
\begin{equation}\label{coi}
\sigma\circ \delta\,=\, \text{Id}_{\TlogD}\, ,
\end{equation}
where $\Atlog$ and
$\sigma$ are constructed in \eqref{e7}. We note that \eqref{coi} implies that
the image $\delta(\TlogD)$ is a holomorphic line subbundle of $\Atlog$.

Similarly, we can construct the logarithmic Atiyah bundle $\cAtlog$ over
$\mathcal{X}$ for a family $\mathcal{E}\, \longrightarrow\, (\mathcal{X},\,
\mathcal{D})\, \longrightarrow\, \Par$ as in the introduction. Note that
$\cTlogD$ fits in the short exact sequence
$$
0 \, \longrightarrow\,\cTlogD \, \longrightarrow\, \cTX\, \longrightarrow\, 
N_{\mathcal D} \, \longrightarrow\, 0\, ,
$$
where $N_{\mathcal D}$ is the normal bundle to $\mathcal D$. So, in this case $\cTX\otimes 
\mathcal{O}_{\mathcal{X}}(-\mathcal{D})$ is a subsheaf of $\cTlogD$.

Again, a logarithmic connection $\nabla$ on $\mathcal{E}$ singular over $\mathcal{D}$ is a 
splitting of the corresponding Atiyah exact sequence. Rather than recalling the definition of 
flatness for $\nabla$ (see for example \cite[p.~131, Lemma 4.1]{BHH}), we just point out the 
following. A logarithmic connection $\delta$ as
in \eqref{coi} is called \textit{irreducible} if there is 
no holomorphic line subbundle $L'$ of $E$ such that $\delta(\TlogD(L'))\, \subset\, L'$. If 
$\nabla$ is flat then the connection $\nabla\vert_{E_s}$ on $E_s\, \longrightarrow\, X_s$ is 
irreducible for some $s\,\in\, \Par$ if and only if $\nabla\vert_{E_t}$ is irreducible for all 
$t\,\in\, \Par$.
 
\subsection{Infinitesimal deformations}\label{se2.2}

As before, $E\, \longrightarrow\, X$ is a rank two holomorphic vector bundle over a
compact Riemann surface $X$ of genus $g\,\geq\, 2$, and $D$ is a reduced effective divisor of degree $n$ on
$X$; let $\delta$ be a logarithmic connection on $E$ with polar divisor $D$.
An infinitesimal deformation of the $n$-pointed Riemann surface $(X,\,D)$ is a family
$$\xymatrix{X\ar[r]^f\ar[d]& \mathcal{X}\ar[d]\\0 \ar[r]&\mathcal{B}:=\mathrm{Spec}(\mathbb{C}[\varepsilon]/\varepsilon^2)}$$
together with a divisor $\mathcal{D}$ on $\mathcal{X}$ given by $n$ disjoint sections
$\mathcal{B}\,\longrightarrow\, \mathcal{X}$ such that $f^*\mathcal{D}\,=\,D$. 
The space of infinitesimal deformations of $(X,\,D)$ is given by $\Coh^1(X,\, \TXD)$, which
coincides with the fiber of the holomorphic tangent bundle
of $\mathrm{Teich}_{g,n}$ at the point corresponding to $(X,\,D)$. Any
such nontrivial deformation naturally embeds into the universal Teichm\"uller curve. The space of
infinitesimal deformations of the triple $(X,\, D,\, E)$ is given by $\Coh^1(X,\, \Atlog)$ \cite[p.~127,
(2.12)]{BHH}. On the other hand, if we have an infinitesimal deformation of
$(X,\, D,\, E)$ that can be endowed with a flat logarithmic connection $\nabla$, inducing $\delta$
on the central parameter, then by the universal property of the universal isomonodromic
deformation, the deformation of $(X,\, D,\, E)$ is already determined by the induced deformation
of $(X,\, D)$. Hence $\delta$ defines a homomorphism of infinitesimal deformations
\begin{equation}\label{id2}
T^\delta\, :\, \Coh^1(X,\, \TXD) \, \longrightarrow\, \Coh^1(X,\, \Atlog)\, .
\end{equation}
 
\begin{lemma}\label{lem3}
The homomorphism $T^\delta$ in \eqref{id2} coincides with the homomorphism
$$
\delta_*\, :\, \Coh^1(X,\, \TXD) \, \longrightarrow\, \Coh^1(X,\, \Atlog)
$$
associated to the connection homomorphism $\delta\, :\, \TXD \, \longrightarrow\, \Atlog$.
\end{lemma}

\begin{proof}
Using the notation of
\cite[p.~131, Lemma 4.1]{BHH} and the commutative diagram in
\cite[p.~128, Section 2]{BHH}, we have the commutative diagram of homomorphisms of
sheaves
\begin{equation}\label{diag}
\xymatrix{
0\ar[r] & \TXD\ar[r]\ar[d]^{{\delta}}& f^*\cTlogD\ar[d]^{\nabla} \ar[r]
 & {\mathcal O}_{X}\ar[r]\ar@{=}[d]& 0\\
0\ar[r] &\Atlog \ar[r] & f^* \cAtlog \ar[r] & {\mathcal O}_{X}\ar[r] & 0
 }
\end{equation}
{}From \eqref{diag} we have the commutative diagram of cohomologies
\begin{equation}\label{diag2}
\xymatrix{
{\mathbb C}\,=\, \Coh^0(X,\, {\mathcal O}_{X}) \ar[r]^a\ar@{=}[d]&\Coh^1(X,\, \TXD)\ar[d]^{\delta_*} \\
 {\mathbb C}\,=\, \Coh^0(X,\, {\mathcal O}_{X}) \ar[r]^b &
\Coh^1(X,\, \Atlog)\, ,
 }
\end{equation}
where $a$ (respectively, $b$) is the connecting homomorphism in the long exact
sequence of cohomologies associated to the top (respectively, bottom) exact
sequence in \eqref{diag}.

Now, $a(1)\, \in\, \Coh^1(X,\, \TXD)$ is the infinitesimal deformation class
for the family of $n$-pointed Riemann surfaces $({\mathcal X},\, {\mathcal D})$
\cite[p.~125, (2.4)]{BHH}, and $$b(1)\, \in\, \Coh^1(X,\, \Atlog)$$
is the infinitesimal deformations of $(X,\, D,\, E)$ associated to the
isomonodromic deformation of $\delta$ \cite[p.~132, Lemma 4.2]{BHH}. Therefore, the
proof is completed by the commutativity of~\eqref{diag2}.
\end{proof}

\section{Higgs bundles on a fixed curve}

In this section, we recall the deformation theory of Higgs bundles over a fixed
Riemann surface.
 
\subsection{Stable Higgs bundles}

Let $X$ be a compact connected Riemann surface of genus $g$, with $g\, \geq\, 2$.
For any integer $d$, let ${\mathcal M}^d_X$ denote the moduli space of stable
vector bundles on $X$ of rank two and degree $d$. It is an irreducible smooth
quasi-projective variety defined over $\mathbb C$ of dimension $4g-3$. The locus
of very stable bundles in ${\mathcal M}^d_X$ is a nonempty Zariski open
subset \cite{La}, \cite[p. 229, Corollary 5.6]{BR}.

Fix an integer $d$. Let
$$
{\mathcal M}_X\, :=\, {\mathcal M}^d_X
$$
be the moduli space of stable
vector bundles on $X$ of rank two and degree $d$. Let
\begin{equation}\label{e1}
{\mathcal U}_{\rm vs}\, \subset\, {\mathcal M}_X
\end{equation}
be the moduli space of very stable vector bundles on $X$ of rank two and degree $d$.
As noted before, ${\mathcal U}_{\rm vs}$ is a
nonempty Zariski open subset of ${\mathcal M}_X$. Let
\begin{equation}\label{e0}
{\mathcal W} \, :=\, {\mathcal M}_X \setminus {\mathcal U}_{\rm vs}
\end{equation}
be the complement consisting of wobbly bundles, which is a closed sub-scheme of
${\mathcal M}_X$.

For any $E\, \in\, {\mathcal M}_X$, we have $\Tan^*_E {\mathcal M}_X\,=\,
\Coh^0(X,\, \text{End}(E)\otimes \KX)$, and hence
the total space $\Tan^*{\mathcal M}_X$ of the holomorphic cotangent bundle of
${\mathcal M}_X$ in \eqref{e1} is a Zariski open dense subset of the moduli
space of stable Higgs bundles of rank two and degree $d$ (openness follows
from \cite{Maru}). Let
\begin{equation}\label{e-1}
{\mathcal N}_X\, \subset\, \Tan^*{\mathcal M}_X
\end{equation}
be the locus of all $(E,\, \theta)$ such that $\theta$ is nonzero nilpotent; this
${\mathcal N}_X$ is a locally closed sub-scheme of $\Tan^*{\mathcal M}_X$. The closure of
${\mathcal N}_X$ in $\Tan^*{\mathcal M}_X$ is the union ${\mathcal N}_X \bigcup{\mathcal
W}$, where ${\mathcal W}$ is the wobbly locus defined in \eqref{e0}.

\subsection{Infinitesimal deformations of Higgs bundles}

Take any Higgs bundle $(E,\, \theta)\, \in\, \Tan^* {\mathcal M}_X$. Consider the
${\mathcal O}_X$--linear homomorphism
$$
f_\theta\, :\, \text{End}(E)\,=\, E\otimes E^*\, \longrightarrow\,
\text{End}(E)\otimes \KX
$$
defined by $A\, \longmapsto\, \theta\circ A- A\circ\theta$. This produces the following
two--term complex ${\mathcal C}^{(E,\theta)}_{\bullet}$ on $X$
\begin{equation}\label{e2}
{\mathcal C}^{(E,\theta)}_{0}\,=\, \text{End}(E) \, \stackrel{f_\theta}{\longrightarrow}\,
{\mathcal C}^{(E,\theta)}_{1}\,=\, \text{End}(E)\otimes \KX\, .
\end{equation}
The space of all infinitesimal deformations of the Higgs bundle $(E,\, \theta)$ is
parametrized by the first hypercohomology ${\mathbb H}^1({\mathcal C}^{(E,\theta)}_{\bullet})$
\cite{BR}, \cite{Ma}, \cite{Bot}. In particular, we have
$$
\Tan_{(E,\theta)}(\Tan^*{\mathcal M}_X)\,=\, {\mathbb H}^1({\mathcal C}^{(E,\theta)}_{\bullet})\, .
$$

Now assume that $(E,\, \theta)\, \in\, {\mathcal N}_X$, where ${\mathcal N}_X$ is
constructed in \eqref{e-1}. Note that this implies
that $E\,\in\, {\mathcal W}$. The subsheaf of $E$ defined by the kernel of $\theta$ is
a line subbundle of $E$; this line subbundle of $E$ will be denoted by $L$. Let
$$
\text{End}^s(E)\, \subset\, \text{End}(E)
$$
be the subbundle of rank three given by the kernel of the natural homomorphism
$$
{\rm End}(E)\, \longrightarrow\, \text{Hom}(L,\, E/L)
$$
that sends any locally defined endomorphism $A$ of $E$ to
the composition
$$
L\, \hookrightarrow\, E \,\stackrel{A}{\longrightarrow}\, E \,\longrightarrow\, E/L\, ;
$$
in other words, $\text{End}^s(E)$ is the subsheaf of endomorphisms of $E$
that preserve the line subbundle $L$. So we have a homomorphism
${\rm End}^s(E)\, \longrightarrow\, \text{Hom}(L,\, L)$ that sends
any locally defined section $A$ of $\text{End}^s(E)$ to $A\vert_L$. This
also implies that we have a homomorphism
${\rm End}^s(E)\, \longrightarrow\, \text{Hom}(E/L,\, E/L)$. Let
\begin{equation}\label{endn}
\text{End}^n(E)\, \subset\, \text{End}^s(E)
\end{equation}
be the line subbundle given by the kernel of the homomorphism
$$
{\rm End}^s(E)\, \longrightarrow\, \text{Hom}(L,\, L)\oplus \text{Hom}(E/L,\, E/L)
$$
constructed above, so $\text{End}^n(E)$ consists of nilpotent endomorphisms of $E$
with respect to the filtration $0\,\subset\, L\, \subset\, E$; the superscripts ``$s$'' and ``$n$''
stand for ``solvable'' and ``nilpotent'' respectively.

For the homomorphism $f_\theta$ in \eqref{e2} we have
\begin{equation}\label{hin}
f_\theta (\text{End}^s(E))\, \subset\, \text{End}^n(E)\otimes \KX\, .
\end{equation}
The restriction of $f_\theta$ to $\text{End}^s(E)$ will be denoted by
$f^s_\theta$. We have the two--term complex ${\mathcal D}^{(E,\theta)}_{\bullet}$ on $X$
\begin{equation}\label{e5}
{\mathcal D}^{(E,\theta)}_{0}\,=\, \text{End}^s(E) \, \stackrel{f^s_\theta}{\longrightarrow}\,
{\mathcal D}^{(E,\theta)}_{1}\,=\,\text{End}^n(E)\otimes \KX
\end{equation}
which is a subcomplex of ${\mathcal C}^{(E,\theta)}_{\bullet}$ constructed
in \eqref{e2}.

Denote the degree of the line bundle $L$ by $\nu$. Consider ${\mathcal N}_X$
in \eqref{e-1}; let
$$
{\mathcal Z}^\nu \, \subset\, {\mathcal N}_X
$$
be the locus of all $(E',\, \theta')\,\in\, {\mathcal N}_X$
such that $\degree(\text{kernel}(\theta'))
\,=\, \nu$. So we have $(E,\,\theta)\, \in\, {\mathcal Z}^\nu$. The tangent
space to ${\mathcal Z}^\nu$ at this point $(E,\,\theta)$ has the following description:
$$
\Tan_{(E,\theta)}{\mathcal Z}^\nu\, =\, {\mathbb H}^1({\mathcal D}^{(E,\theta)}_{\bullet})\, ,
$$
where ${\mathcal D}^{(E,\theta)}_{\bullet}$ is the complex in \eqref{e5} \cite[p.~228]{BR}.

\section{Higgs bundles on a family of curves}

We now recall the deformation theory of families of Higgs bundles over moving 
curves. As noted in Section \ref{se2.2}, infinitesimal deformations of a
triple $(X,\, D,\, E)$ is given by the first cohomology space of the logarithmic
Atiyah bundle associated to this triple. Now we have to enrich this space with the data 
corresponding to deformations of a given Higgs field on $E$. Afterwards, 
we will calculate the obstruction space for an initial nonzero nilpotent Higgs field to 
extend to a nilpotent Higgs field on a germification of the family.

\subsection{Infinitesimal deformation of a $n$-pointed curve and Higgs bundle}

Let $E$ be a rank $2$ vector bundle on $X$ of genus $g\, \geq\, 2$, and let
$D\,=\,\sum_{i=1}^n x_i$ be a divisor on $X$.

There is a natural homomorphism
\begin{equation}\label{eta}
\eta\, :\, \Atlog \, \longrightarrow\, \text{Diff}^1(\text{End}(E)\otimes \KX,\,
\text{End}(E)\otimes \KX)\, ,
\end{equation}
where $\text{Diff}^1(\text{End}(E)\otimes \KX,\,\text{End}(E) \otimes \KX)$ is the 
vector bundle on $X$ defined by the sheaf of differential operators of order at most one 
mapping locally defined holomorphic sections of $\text{End}(E) \otimes \KX$ to 
itself. To explain this $\eta$, if
\begin{itemize}
\item $\alpha$ is a locally defined holomorphic section of $ \Atlog$,

\item $\beta$ is a locally defined holomorphic section of $\text{End}(E)$,

\item $\omega$ is a locally defined holomorphic section of $\KX$ and

\item $s$ is a locally defined holomorphic section of $E$,
\end{itemize}
then
\begin{equation}\label{eta2}
(\eta(\alpha))(\beta\otimes\omega)(s)\,=\, \alpha(\beta(s))\otimes\omega +
\beta(s)\otimes (L_{\sigma (\alpha)}\omega) - \beta(\alpha(s))\otimes\omega\, ,
\end{equation}
where $L_{\sigma (\alpha)}\omega$ is the Lie derivative of $\omega$ with respect
to the vector field $\sigma (\alpha)$ (the homomorphism $\sigma$ is defined
in \eqref{e7}); note that both sides of \eqref{eta2} are
sections of $E\otimes \KX$. To prove that \eqref{eta2} defines $\eta$, substitute
$(f\beta)\otimes (\frac{1}{f}\omega)$ in place of $\beta\otimes\omega$ in
\eqref{eta2}, where $f$ is a locally defined nowhere vanishing holomorphic function on $X$. Then
in the right--hand side of \eqref{eta2}, the first term becomes
$\alpha(\beta(s))\otimes\omega+ \frac{\mathrm{d}f(\sigma (\alpha))}{f}\cdot \beta(s)\otimes\omega$
while the second term becomes $\beta(s)\otimes (L_{\sigma (\alpha)}\omega)
-\frac{\mathrm{d}f(\sigma (\alpha))}{f}\cdot \beta(s)\otimes\omega$. Therefore, we have
$(\eta(\alpha))(\beta\otimes\omega)\,=\, (\eta(\alpha))((f\beta)\otimes
(\frac{1}{f}\omega))$. From this it follows immediately
that $\eta$ is well--defined by \eqref{eta2}.

We will give another description of the homomorphism $\eta$ which is more canonical.

Let $p\, :\, P_{{\rm GL}(2)}\, \longrightarrow\, X$ be the holomorphic principal 
$\text{GL}(2, {\mathbb C})$--bundle on $X$ corresponding to $E$; so for any 
$x\,\in\, X$, the fiber $p^{-1}(x)$ is the space of all $\mathbb C$--linear
isomorphisms from 
${\mathbb C}^{\oplus 2}$ to the fiber $E_x$. Let
\begin{equation}\label{dph}
\mathrm{d}p\, :\, \Tan P_{{\rm GL}(2)}\, 
\longrightarrow\, p^*\TX
\end{equation}
be the differential of the above projection $p$. The kernel
$$
\Tan_{\rm rel} \, :=\, \text{kernel}(\mathrm{d}p)\, \subset\, \Tan P_{{\rm GL}(2)}
$$
is identified with $p^* \text{End}(E)$. The pullback $p^*\Atlog$ is
identified with $(\mathrm{d}p)^{-1}(\TXD)\, \subset\, \Tan P_{{\rm GL}(2)}$. The
image of $p^*\KX$ under the dual homomorphism $(\mathrm{d}p)^*\, :\, p^*\KX \,
\longrightarrow\, \Tan^*P_{{\rm GL}(2)}$ (see \eqref{dph}) will be denoted by
$\mathcal S$. For any
\begin{itemize}
\item open subset $U\, \subset\, X$,

\item $\text{GL}(2, {\mathbb C})$--invariant holomorphic section $\alpha$ of
$(\mathrm{d}p)^{-1}(\TXD)\,=\, p^*\Atlog$ defined over $p^{-1}(U)$, and

\item $\text{GL}(2, {\mathbb C})$--invariant holomorphic section $\beta$ of 
$\Tan_{\rm rel}\otimes {\mathcal S}$ over $p^{-1}(U)$,
\end{itemize}
the Lie derivative $L_\alpha \beta$ is again a $\text{GL}(2, {\mathbb C})$--invariant
holomorphic section of $\Tan_{\rm rel}\otimes {\mathcal S}$ over $p^{-1}(U)$. 
The homomorphism $\Atlog\, \longrightarrow\, \text{Diff}^1(\text{End}(E) 
\otimes \KX,\, \text{End}(E)\otimes \KX)$ defined by $\alpha\, \longmapsto\,
\{\beta\, \longmapsto\, L_\alpha \beta\}$ coincides with $\eta$.

Let $\theta$ be a Higgs field on the vector bundle $E$. Let
\begin{equation}\label{eth}
\psi_\theta\, :\, \Atlog\,\longrightarrow\,\text{End}(E)\otimes \KX\, , \ \ \ 
\alpha\, \longmapsto\, \eta(\alpha)(\theta)
\end{equation}
be the ${\mathcal O}_X$--linear homomorphism, where $\eta$ is the homomorphism
in \eqref{eta}.

\begin{lemma}[{\cite[p.~105,~Proposition~2.4]{Bi}}]\label{lem1}
The space of infinitesimal deformations of a $n$--pointed Riemann surface equipped
with a Higgs bundle
$$
\underline{z}\,=\, (X,\,D,\, (E,\, \theta))
$$
as above is the first hypercohomology ${\mathbb H}^1({\mathcal A}^{\underline{z}}_{\bullet})$,
where ${\mathcal A}^{\underline{z}}_{\bullet}$ is the two--term complex
$$
{\mathcal A}^{\underline{z}}_{0}\,=\, \Atlog \,
\stackrel{\psi_\theta}{\longrightarrow}\, {\mathcal A}^{\underline{z}}_1\,=\,
{\rm End}(E)\otimes \KX\, ,
$$
where $\psi_\theta$ is constructed in \eqref{eth}.
\end{lemma}

Note that in the case where $\theta$ is the zero--Higgs field, the space of infinitesimal 
deformations in Lemma \ref{lem1} coincides with $\Coh^1(X,\,\Atlog )\oplus \Coh^0(X,\, {\rm 
End}(E)\otimes \KX)$; as shown in Section \ref{Sec:Atiyah}, the space of infinitesimal 
deformations of the triple $(X,\,D,\, E)$ is $\Coh^1(X,\, \Atlog)$.

\subsection{Deformation of a $n$-pointed curve with a nilpotent Higgs 
bundle}\label{Sec:DefNilHiggs}

Consider the data $\underline{z}\,=\, (X,\,D,\, (E,\, \theta))$
in Lemma \ref{lem1}. Assume that $(E,\, \theta)\, \in\, {\mathcal N}_X$, where ${\mathcal N}_X$
is defined in \eqref{e-1}. As before, the line subbundle of $E$ defined by the kernel
of $\theta$ will be denoted by $L$. Let
$$
\AtEL\,\subset\, \Atlog
$$
be the holomorphic subbundle of co-rank one generated by the sheaf of differential operators
$\alpha$ such that $\alpha(L)\, \subset\, L$. Note that we have a commutative diagram
\begin{equation}\label{cd}
\xymatrix{
& 0\ar[d] & 0\ar[d] &0\ar[d] \\
 0 \ar[r] & \text{End}^s(E) \ar[r]\ar[d]& \AtEL \ar[r]\ar[d]^{\iota}
 & \TXD \ar[r]\ar@{=}[d]& 0\\
0 \ar[r] & \text{End}(E) \ar[r]\ar[d]& \Atlog \ar[r]^{\sigma}\ar[d]^q
 & \TXD \ar[r] & 0\\
& \text{Hom}(L,\, E/L) \ar@{=}[r]\ar[d]& \text{Hom}(L,\, E/L) \ar[d]\\
 & 0 & 0
}
\end{equation}
where all the rows and columns are exact; the bottom exact row is the one
in \eqref{e7} and $\text{End}^s(E)$ is the vector bundle in \eqref{endn}.
It can be shown that the homomorphism $\psi_\theta$ in \eqref{eth} satisfies the equation
$$
\psi_\theta(\AtEL)\, \subset\, \text{End}^s(E)\otimes \KX\, .
$$
Indeed, this follows from the expression in the right--hand side of \eqref{eta2}.

The restriction of $\psi_\theta$ to $\AtEL$ will also be denoted by 
$\psi_\theta$; this should not cause any confusion.

\begin{lemma}\label{lem2}
The space of infinitesimal deformations of the $n$--pointed Riemann surface equipped
with a nilpotent Higgs bundle
$$
\underline{z}\,=\, (X,\, D,\, (E,\, \theta))\, ,
$$
such that the Higgs field remains nilpotent,
is the first hypercohomology ${\mathbb H}^1({\mathcal B}^{\underline{z}}_{\bullet})$,
where ${\mathcal B}^{\underline{z}}_{\bullet}$ is the two--term complex
$$
{\mathcal B}^{\underline{z}}_{0}\,=\, \AtEL \,
\stackrel{\psi_\theta}{\longrightarrow}\, {\mathcal B}^{\underline{z}}_1\,=\,
{\rm End}^s(E)\otimes \KX\, .
$$
\end{lemma}

\begin{proof}
Consider the short exact sequence of complexes
\begin{equation}\label{g1}
\xymatrix{
0 \ar[d]&& 0\ar[d]\\
{\mathcal B}^{\underline{z}}_{0}\,=\, \AtEL \ar[rr]^{\hskip-10pt \psi_\theta} \ar[d] && {\mathcal B}^{\underline{z}}_1\,=\,
{\rm End}^s(E)\otimes \KX \ar[d] \\
{\mathcal A}^{\underline{z}}_{0}\,=\, \Atlog \ar[rr]^{\hskip-10pt \psi_\theta} \ar[d]&& {\mathcal A}^{\underline{z}}_1\,=\,
{\rm End}(E)\otimes \KX \ar[d]\\
{\rm End}(E)/{\rm End}^s(E)\ar[rr]^{\hskip-10pt f'_\theta} \ar[d] && ({\rm End}(E)/{\rm End}^s(E))\otimes \KX\ar[d]\\ 
0 && 0
 }
\end{equation}
where the middle complex is the one in Lemma \ref{lem1}, and the
homomorphism $f'_\theta$ is induced by $f_\theta$ in \eqref{e2}; note
that from \eqref{hin} it follows immediately that $f_\theta$ induces such a homomorphism.

Consider the above complex $\widetilde{\mathcal C}_{\bullet}$
$$
\widetilde{\mathcal C}_0 \,=\, {\rm End}(E)/{\rm End}^s(E) \,\stackrel{f'_\theta
}{\longrightarrow} \,\widetilde{\mathcal C}_1\,=\,({\rm End}(E)/{\rm End}^s(E))\otimes \KX\, .
$$
The homomorphism ${\mathbb H}^1({\mathcal 
A}^{\underline{z}}_{\bullet})\,\longrightarrow\, {\mathbb H}^1(\widetilde{\mathcal 
C}_{\bullet})$ induced by the homomorphism of complexes in \eqref{g1} coincides
with the natural homomorphism from the infinitesimal deformations of
$(X,\, D,\, (E,\, \theta))$ (without assuming that
the Higgs field remains nilpotent) to the failure of the Higgs field to be nilpotent. The
lemma follows from it.
\end{proof} 

\section{Proof of Theorem \ref{thm1}}

Let $E$ be a holomorphic vector bundle on $X$ of rank two and degree $d$, which
is wobbly. Let $\theta$ be a nonzero nilpotent Higgs field on $E$. As before, let $L\,
\subset\, E$ be the line subbundle defined by the kernel of $\theta$. Given a
logarithmic connection $\delta$ on $E$ with polar divisor $D$,
consider the composition homomorphism
\begin{equation}\label{c3}
q\circ\delta\,:\,\TXD \, \longrightarrow\, \text{Hom}(L,\, E/L)\, ,
\end{equation}
where $q$ is the homomorphism in \eqref{cd}. Let
\begin{equation}\label{d3}
(q\circ\delta)_*\, :\, \Coh^1(X,\, \TXD) \, \longrightarrow\, \Coh^1(X,\,
\text{Hom}(L,\, E/L))
\end{equation}
be the homomorphism of cohomologies induced by $q\circ\delta$ in \eqref{c3}.

\begin{proposition}\label{prop1}
Let $(\mathcal{X},\, \mathcal{D},\, \mathcal{E})$ be an infinitesimal deformation,
with parameter space $\mathcal{B}\,=\,\mathrm{Spec}\left(\mathbb{C}[\varepsilon]/
\varepsilon^2\right)$, of a stable rank two vector bundle $E$ over a
$n$-pointed Riemann surface $(X,\,D)$. Assume that $\mathcal{E}$ is endowed with
a nonzero nilpotent Higgs field $\Theta$, inducing a nonzero nilpotent Higgs field
$\theta$ on $E$. Further assume that $\mathcal{E}$ is endowed with a flat logarithmic
connection $\nabla$, singular over $\mathcal{D}$, inducing a logarithmic connection
$\delta$ on $E$ with polar divisor $D$. Then
$$
(q\circ\delta)_*\, =\, 0\, ,
$$
where $(q\circ\delta)_*$ is the homomorphism in \eqref{d3}.
\end{proposition}

\begin{proof}
Consider the complex ${\mathcal B}^{\underline{z}}_{\bullet}$ in Lemma \ref{lem2}.
The identity map of ${\mathcal B}^{\underline{z}}_{0}$ produces a homomorphism
of complexes
$$
\xymatrix{
{\mathcal B}^{\underline{z}}_{0}\,=\, \AtEL \ar[rr]^{\hskip25pt \psi_\theta} \ar[d] && {\mathcal B}^{\underline{z}}_1\ar[d]\\
\AtEL\ar[rr]&& 0
 }
$$
The $i$-th hypercohomology of the bottom complex coincides with
$\Coh^i(X,\, \AtEL)$. Let 
\begin{equation}\label{tau}
\phi\, :\, {\mathbb H}^1({\mathcal B}^{\underline{z}}_{\bullet})\,\longrightarrow\,
\Coh^1(X,\, \AtEL)
\end{equation}
be the homomorphism of hypercohomologies associated to the
above homomorphism of complexes.

Since the nilpotent Higgs field $\theta$ on $E$ extends as a nilpotent Higgs field 
$\Theta$ on $\mathcal{E}$, from Lemma \ref{lem2} it follows that the homomorphism 
$T^\delta$ in \eqref{id2} factors as
\begin{equation}\label{tau2}
T^\delta\,=\, \iota_*\circ\phi\circ T^\delta_0\, ,
\end{equation}
where
$$
T^\delta_0\, :\, \Coh^1(X,\, \TXD) \, \longrightarrow\,
{\mathbb H}^1({\mathcal B}^{\underline{z}}_{\bullet})
$$
is a homomorphism, $\phi$ is the homomorphism in \eqref{tau} and
$$
\iota_*\, :\, \Coh^1(X,\, \AtEL)\, \longrightarrow\,
\Coh^1(X,\, \Atlog)
$$
is the homomorphism of cohomologies induced by the inclusion
$\iota\, :\, \AtEL \, \hookrightarrow\, \Atlog$
(see \eqref{cd}).

{}From Lemma \ref{lem3} we know that
\begin{equation}\label{lk}
(q\circ\delta)_*\,=\, q_*\circ T^\delta\, ,
\end{equation}
where
$$
q_*\, :\, \Coh^1(X,\, \Atlog)\, \longrightarrow\, \Coh^1(X,\, 
\text{Hom}(L,\, E/L))
$$
is the homomorphism of cohomologies induced by the homomorphism $q$ in
\eqref{cd}. Now from \eqref{tau2} and \eqref{lk} it follows that
$$
(q\circ\delta)_* \,=\, q_*\circ\iota_*\circ\phi\circ T^\delta_0\, .
$$
On the other hand, from \eqref{cd} it follows immediately that
$q_*\circ\iota_*\,=\, 0$. Hence we conclude that $(q\circ\delta)_* \,=\,0$.
\end{proof}

Let $\mathrm{Teich}_{g,n}$ be the Teichm\"uller space for $n$-pointed surfaces of
genus $g\, \geq\, 2$. Take $(X,\, D,\, E,\, \delta)$ as before.
Let $(\mathcal{X},\, \mathcal{D},\, \mathcal{E}\, , \nabla)$ be the universal
isomonodromic deformation of $(X,\, D,\, E,\, \delta)$, with parameter space 
$\mathcal{T}\,=\,\mathrm{Teich}_{g,n}$ and central parameter $t_0\,\in \,\mathcal{T}$, 
together with an isomorphism of $n$-pointed surfaces $f \,:\,(X,\,D)\,\longrightarrow\,
(X_{\para_0},\, D_{\para_0})$ and a holomorphic isomorphism $\psi \,:\, E\,
\longrightarrow\, f^*\mathcal{E}$ such that $f^*\nabla \,=\,\psi_*(\delta)$
as in the introduction. Note that for any $t_1\,\in\, \mathcal{T}$, the universal
isomonodromic deformation $(\mathcal{X} ,\, \mathcal{D},\,\mathcal{E},\, \nabla)$ is
also the universal isomonodromic deformation of
$(\mathcal{X} ,\, \mathcal{D},\,\mathcal{E},\, \nabla)\vert_{t_1}$. As before, for
any $t\,\in\, \mathcal{T}$, we shall denote 
$(\mathcal{X},\, \mathcal{E})\vert_{t}$ by $(X_t,\, E_t)$.

In view of Theorem \ref{BHHrank2} and the openness of the very stability 
condition, the following theorem implies Theorem \ref{thm1}.
 
\begin{theorem}\label{thm3}
Let $(X,\, D,\, E,\, \delta)$ and $(\mathcal{X},\, \mathcal{D},\,
\mathcal{E}\, , \nabla)$ be as above with $\delta$ irreducible. Then there is a point 
$t\,\in\, \mathcal{T}$ such that the vector bundle $E_t\longrightarrow {X}_t$ is
very stable. 
\end{theorem}

\begin{proof}
In view of Theorem \ref{BHHrank2} we can take $E$ to be stable. Assume
that $E$ is wobbly. Let $\theta$ be a nonzero nilpotent Higgs field on
$E$. In view of Proposition \ref{prop1} it suffices to show that
\begin{equation}\label{c2}
(q\circ\delta)_*\, \not=\, 0\, ,
\end{equation}
where $(q\circ\delta)_*$ is the homomorphism in \eqref{d3}.
Note that the homomorphism $q\circ\delta$ in \eqref{c3} is nonzero because
the logarithmic connection $\delta$ is irreducible. Consider the short
exact sequence of coherent sheaves
$$
0\,\longrightarrow\, \TXD \, \stackrel{q\circ\delta}{\longrightarrow}\,
\text{Hom}(L,\, E/L) \,\longrightarrow\, \mathbb{T} \,\longrightarrow\, 0\, ,
$$
where $\mathbb{T}$ is a torsion sheaf on $X$, so $\Coh^1(X,\, \mathbb{T})\,=\,0$.
Hence the long exact sequence of cohomologies associated to it gives a surjection
$$
\Coh^1(X,\, \TXD) \, \stackrel{(q\circ\delta)_*}{\longrightarrow}\, \Coh^1(X,\, 
\text{Hom}(L,\, E/L))\,\longrightarrow\, 0\, .
$$
Therefore, to prove \eqref{c2} it suffices to show that
\begin{equation}\label{c4}
\Coh^1(X,\, \text{Hom}(L,\, E/L))\, \not=\, 0\, .
\end{equation}
Now, by Serre duality,
$$
\Coh^1(X,\, \text{Hom}(L,\, E/L))\,=\, \Coh^0(X,\, \text{Hom}(E/L,\, L)\otimes \KX)^*\, .
$$
Since $\theta$ is nonzero nilpotent and $L\, \subset\, E$ is the kernel of $\theta$, it
follows immediately that $\theta$ is a nonzero section of $\text{Hom}(E/L,\, L)\otimes
\KX$. Therefore, \eqref{c4} is proved. This completes the proof of the theorem.
\end{proof}

When the rank is more than two, the very last part of the argument in
the proof of Theorem \ref{thm3} breaks down --- the nonzero nilpotent Higgs field
$\theta$ no longer implies that $(q\circ\delta)_*\, \not=\, 0$.


\end{document}